\documentclass[12pt]{amsart}

\usepackage[english]{babel}
\usepackage[utf8x]{inputenc}
\usepackage[T1]{fontenc}
\usepackage[normalem]{ulem}

\usepackage{fullpage}

\usepackage{amsmath}
\usepackage{amsfonts}
\usepackage{amsthm}
\usepackage{graphicx}
\usepackage{enumerate}
\usepackage[colorinlistoftodos]{todonotes}
\usepackage[colorlinks=true, allcolors=blue]{hyperref}

\numberwithin{equation}{section}
\newtheorem{theorem}{Theorem}[section]
\newtheorem{corollary}[theorem]{Corollary}
\newtheorem{lemma}[theorem]{Lemma}
\newtheorem{proposition}[theorem]{Proposition}
\newtheorem{question}[theorem]{Question}
\newtheorem{conjecture}[theorem]{Conjecture}
\newtheorem{claim}[theorem]{Claim}
\theoremstyle{definition}
\newtheorem{definition}[theorem]{Definition}
\theoremstyle{remark}
\newtheorem{remark}[theorem]{Remark}
\newtheorem{example}[theorem]{Example}

\newcommand{\N}{\mathbb{N}}
\newcommand{\R}{\mathbb{R}}
\newcommand{\Z}{\mathbb{Z}}
\newcommand{\bfb}{{\mathbf {B}}}
\newcommand{\ov}{\overline}
\newcommand{\ch}{\mathbf 1}
\newcommand{\spa}{\operatorname{span}}

\title{Simultaneous dilation and translation tilings of $\R^n$}
\author{Marcin Bownik \and Darrin Speegle}

\thanks{The first author was partially supported by NSF grant DMS-1956395. The authors wish to thank Barak Weiss for showing us the long history of connections between Diophantine approximation and ergodic theory.}


\begin{document}

\begin{abstract}
We solve the wavelet set existence problem. That is, we characterize the full-rank lattices $\Gamma\subset \R^n$ and invertible $n \times n$ matrices $A$ for which there exists a measurable set $W$ such that $\{W + \gamma: \gamma \in \Gamma\}$ and $\{A^j(W): j\in \Z\}$ are tilings of $\R^n$.  The characterization is a non-obvious generalization of the one found by Ionascu and Wang \cite{IonWan06}, which solved the problem in the case  $n = 2$.  As an application of our condition and a theorem of Margulis, we also strengthen a result of Dai, Larson, and the second author on the existence of wavelet sets by showing that wavelet sets exist for matrix dilations, all of whose eigenvalues $\lambda$ satisfy $|\lambda| \ge 1$. As another application, we show that the Ionascu-Wang characterization characterizes those dilations whose product of two smallest eigenvalues in absolute value is $\ge 1$. 
\end{abstract}

\maketitle

\section{Introduction}

We study simultaneous tilings of $\R^n$ by two actions which, on their face, do not have any relationship. The first action is via translation by a full rank lattice $\Gamma \subset \R^n$. There always exists a set $V \subset \R^n$ of finite measure such that $\{V + \gamma: \gamma \in \Gamma\}$ is a measurable tiling of $\R^n$. The second action is via multiplication by integer powers of an invertible matrix $A$. If $\left|\det A\right| \not= 1$, then there exists a set $U$ of finite measure such that $\{A^j(U): j\in \Z\}$ is a measurable tiling of $\R^n$, see \cite{LarSchSpeTay06}.
The question solved in this paper has been explicitly posed by Wang \cite{Wan02, IonWan06} and the second author \cite{Spe03}, although it has been studied earlier in the late 1990s \cite{DaiLarSpe97, Spe97}.

\begin{question}\cite{IonWan06, Spe97} For which pairs $(A, \Gamma)$ does there exist a measurable set $W \subset \R^n$ such that 
\begin{equation}\label{dilationTile}
\{A^j(W): j\in \Z\} \,\, {\text{is a measurable tiling of }} \R^n
\end{equation}
and
\begin{equation}\label{translationTile}
\{W + \gamma: \gamma \in \Gamma\}  \,\, {\text{is a measurable tiling of }} \R^n?
\end{equation}
A set $W$ that satisfies \eqref{dilationTile} and \eqref{translationTile} is called an $(A, \Gamma)$ {\emph {wavelet set}}.
\end{question}

\subsection{Motivation}
Our motivation for studying this problem is three-fold. First, whenever a set $W \subset \R^n$ satisfies \eqref{dilationTile} and \eqref{translationTile}, the  indicator function $\ch_W$ is the Fourier transform of an orthogonal wavelet \cite{DL98, HW96, ILP98}. That is, the inverse Fourier transform $\psi =  \check\ch_W$ generates a wavelet system
\[
\left\{\left |\det B \right|^{j/2} \psi(B^j x + k): j\in \Z, k \in \Gamma^*\right\},
\]
which forms an orthogonal basis for $L^2(\R^n)$, where $B$ is the transpose of $A$ and $\Gamma^*$ is the dual lattice of $\Gamma$. In dimension 2 and higher, it is an open problem to determine for which pairs $\left(B, \Gamma^*\right)$, there exists a $(B, \Gamma^*)$ orthonormal wavelet, see \cite{BowRze, Spe03, Wan02}. The current paper provides many new examples of pairs for which such wavelets exist. For all cases currently known, when there exists an orthonormal wavelet, there also exists an $(A, \Gamma)$ wavelet set. The authors conjecture that this is true in general:

\begin{conjecture}
For each pair $(B, \Gamma^*)$ such that there exists a $(B, \Gamma^*)$ orthonormal wavelet, there exists an $(A, \Gamma)$ wavelet set, where $B$ is the transpose of $A$ and $\Gamma^*$ is the dual lattice of $\Gamma$.
\end{conjecture}

Evidence in favor of this conjecture includes \cite{ChuShi00}, where it is shown in dimension 1 that if $a^j$ is irrational for all $j \in \Z \setminus \{0\}$, then the {\emph {only}} $(a,\Z)$ orthonormal wavelets that exist are those that are supported on wavelet sets. A higher dimensional extension of this result was shown by the first author \cite{Bow3}. An even stronger open problem is to determine whether the support of every $(B, \Gamma^*)$ wavelet contains an $(A, \Gamma)$ wavelet set.  This stronger conjecture is open even in the classical dyadic case: dimension $n =1$, dilation $A=B=2$, and translations along integers. This question was originally posed by Larson  in late 1990's although its official formulation appeared only in \cite{Lar}.  It is shown in \cite{RzeSpe02} that this stronger conjecture is true in the one dimensional dyadic case under the additional assumption that the support $E$ of $\hat \psi$ satisfies $\sum_{j\in\Z} \ch_{E} (2^j x)\le 2$ and $\sum_{k\in \Z} \ch_E(x + k) \le 2$. Moreover, it is also known that Larson's problem has an affirmative answer for MRA wavelets \cite{BowRze}.

Our second motivation comes from general tiling questions. For $\alpha \in SO(n)$, let $\Gamma_{\alpha} = \alpha \Z^n$ be a rotation of the integer lattice $\Z^n$. The (measurable) Steinhaus tiling problem is to determine whether there is a single Lebesgue measurable set $E \subset \R^n$ such that $E$ is a fundamental region for each $\R^n/\Gamma_{\alpha}$. This problem was solved in the negative in dimensions 3 and higher by Kolountzakis and Wolff \cite{KW}, but remains open in dimension 2. However, the existence of non-measurable Steinhaus tilings in $\R^2$ was shown by Jackson and Mauldin \cite{JM, JM2}. The commonality in the two problems is that we have multiple actions on $\R^n$ for which it is easy to see that there are measurable tilings when considered separately, yet it is not at all clear when and whether there is a single measurable set which tiles by both actions simultaneously. 

Our third motivation comes from attempts to solve the wavelet set existence problem itself. We found while working on the problem that in one approach we needed to estimate the cardinality of a ball centered at zero intersected with the image of a lattice: $\#\left|\bfb(0, 1) \cap A^j(\Gamma)\right|$. In another approach, we needed to estimate the subspace measure of a lattice subspace intersected with an image of the ball: $m_d(V \cap A^j(\bfb(0, 1))$, where $V$ is the span of some subset of $\Gamma$. The relationship between these two quantities has a long history, and we were intrigued by the connections between our attempts at a solution to the wavelet set existence problem and these well-studied objects.

\subsection{Prior Results}
 Larson, Schulz, Taylor, and the second author \cite{LarSchSpeTay06} have shown that there exists a set of finite measure $W$ that tiles by dilations \eqref{dilationTile} if and only if $\left|\det A\right| \not= 1$. An immediate corollary of this fact is that no wavelet sets exist when the dilation has determinant 1. An interesting counterpoint to that statement was given by the first author and Lemvig \cite{BowLem}, who showed that whenever $\left|\det A\right| \not= 1$, then for almost every lattice $\Gamma$ there is an $(A, \Gamma)$ wavelet set. When $A$ is expansive, that is, all eigenvalues are bigger than one in modulus, Dai, Larson and the second author \cite{DaiLarSpe97} showed that $(A, \Gamma)$ wavelet sets exist for all lattices $\Gamma$. 
 
The second author \cite{Spe03} provided the first necessary conditions and sufficient conditions on the existence of wavelet sets when the dilation $A$ has eigenvalues both bigger than and less than 1 in modulus. Ionascu and Wang \cite[Theorem 1.3]{IonWan06} extended these results to characterize pairs $(A, \Gamma)$ for which wavelet sets exist in the 2-dimensional case. We reformulate their result as follows.

\begin{theorem}\label{iw_intro}
Let $A$ be $2\times 2$ matrix with $\left |\det A \right|>1$ and let $\Gamma$ be a full rank lattice in $\R^2$. Let $\lambda_1$ and $\lambda_2$ be the eigenvalues of $A$ such that $|\lambda_1| \ge |\lambda_2|$. There exists an $(A,\Gamma)$ wavelet set if and only if 
\begin{enumerate}[(i)]
\item $|\lambda_2| \ge 1$, or 
\item $|\lambda_2|<1$ and
$
\ker (A- \lambda_2 \mathbf I) \cap \Gamma = \{0\}.
$
\end{enumerate}
\end{theorem}

The characterizing condition (ii) in Theorem \ref{iw_intro} has several possible restatements in higher dimensions when the smallest eigenvalue of $A$ is less than one in modulus. For example, these four statements are all equivalent in dimension $n=2$ when $\left|\det A\right|>1$.
\begin{enumerate}
\item\label{cone} for every $R > 0$, $\liminf_{j \to \infty} \# | A^{-j}(\mathbf B(0,R)) \cap \Gamma | = 1$,
\item\label{ctwo} for every $R > 0$, $\liminf_{j \to \infty} \# |A^{-j}(\mathbf B(0,R)) \cap \Gamma |< \infty$,
\item\label{cthree} for every sublattice $\Lambda\subset \Gamma$, if $V = \spa(\Lambda)$ and $d=\dim V$, then
\[
\liminf_{j\to \infty} m_d(A^{-j}(\mathbf B(0,1)) \cap V) <\infty,
\]
where $m_d$ denotes the Lebesgue measure on the subspace $V$,
\item\label{cfour} if $V$ is the space spanned by the eigenvectors associated with eigenvalues less than one in modulus, then $V \cap \Gamma = \{0\}$.
\end{enumerate}

It is relatively easy to see the sequence of implications \eqref{cone} $\implies$ \eqref{ctwo} $\implies$ \eqref{cthree} $\implies$ \eqref{cfour} for all dimensions $n$. The condition \eqref{cone} is a known sufficient condition for the existence of wavelet sets, see  \cite[Theorem 2.5]{IonWan06}.  On the other hand, \eqref{cfour} is a necessary condition for the existence of wavelet sets in light of Theorem \ref{iw_intro}, but as we will see later, it is not sufficient in dimensions $n\ge 3$, see Example \ref{obvious}. Conditions \eqref{ctwo} and \eqref{cthree} are  not equivalent; see \cite{Wei04} for an example of the types of theorems in this area and further references in ergodic theory and \cite{Mos10} for Diophantine approximation results that are written in notation and language closer to that of this paper. See also Section \ref{examples} in this paper for related, explicit examples. {\bf {None}} of these natural extensions are equivalent to the existence of wavelet sets in higher dimensions.

\subsection{Statement of Result}
Our main result answers the wavelet set existence problem by giving a necessary and sufficient condition for the existence of wavelet sets. 

\begin{theorem}\label{mt} Let $A$ be an $n\times n$ matrix with $\left |\det A \right|>1$. Let $\Gamma \subset \R^n$ be a full rank lattice.
Then, there exists an $(A,\Gamma)$ wavelet set if and only if 
\begin{equation}\label{char}
\sum_{j=1}^\infty \frac{1}{\# \left|A^{-j}(\mathbf B(0,1)) \cap \Gamma \right |} =\infty,
\end{equation}
\end{theorem}

One appealing characteristic of Theorem \ref{mt} is that the statement is not split into cases where the dilation $A$ is expansive versus where it is not. For example, when $A$ is expansive, it is known that there exists $J \ge 1$ such that $A^{-j}\left(\bfb(0, 1)\right) \subset \bfb(0, 1)$ for all $j \ge J$. Therefore, we recover that wavelet sets exist when $A$ is expansive as a corollary. As another corollary of our characterization, we deduce that condition \eqref{ctwo} is sufficient and condition \eqref{cthree}  is necessary for the existence of a wavelet set (see Theorem \ref{nexist}), respectively. 

The paper is organized as follows. In Section \ref{sufficiency} we show that condition \eqref{char} is sufficient for the existence of a wavelet set. In Section \ref{necessity} we show the same condition \eqref{char} is also necessary. In Section \ref{oldnecessity}, we show that a weaker, but more easily checked, condition is a necessary condition for the existence of wavelet sets.
In Section \ref{idk}, we examine applications of Theorem \ref{mt}. The first application is Theorem \ref{eigenbiggerone}, where we show the existence of $(A, \Gamma)$ wavelet sets for any lattice $\Gamma$ if all eigenvalues of $A$ are greater than or equal to one in modulus using a theorem due to Margulis \cite{Margulis71}. The second application is a generalization of Theorem \ref{iw_intro} for dilations whose product of two smallest eigenvalues in absolute value is $\ge 1$. The third application is the existence of $(A,\Z^n)$ wavelet sets for matrices $A$ with integer entries. Section \ref{examples} gives examples which illustrate the main theorem. In particular, we provide an example where wavelet sets do not exist even though condition \eqref{cthree} holds.

\section{Definitions and Preliminary Results}

Throughout this paper, we assume that $A$ is an invertible $n \times n$ matrix and $\Gamma$ is a full rank lattice in $\R^n$.  That is, $\Gamma$ is the image of the integer lattice under an invertible linear transform. 

\begin{definition}
Let $M\in \N$. We say a measurable set $U\subset \R^n$ {\it packs $M$-redundantly} by $A$ dilations if 
\[
\sum_{j\in\Z} \ch_U(A^jx) \le M \qquad\text{for a.e. }x\in \R^n.
\]
In the special case that $M$ can be chosen to be 1, we say that $U$ {\it packs} by $A$ dilations.
The set $U$ {\it covers} by $A$ dilations if 
\[
\sum_{j\in \Z} \ch_U(A^jx) \ge 1 \qquad\text{for a.e. }x\in \R^n.
\]
The set $U$ {\it tiles} by $A$ dilations if it both packs and covers, in which case we call $\{A^j(U): j\in \Z\}$ a {\it measurable partition} of $\R^n$, or a \emph{(measurable) tiling} of $\R^n$. 
\end{definition}

Similarly, we say a measurable set $V \subset \R^n$ packs $M$-redundantly, packs, covers, or tiles by $\Gamma$ translations if $\sum_{\gamma \in \Gamma} \ch_V(x + \gamma)$ has the corresponding property. The following is an easy consequence of this definition.

\begin{proposition}\label{easyProp} If a measurable set $U$ packs $M$-redundantly by translations, then there exists a partition $(U_m)_{m=1}^M$ of $U$ into measurable sets such that for each $1\le m \le M$, $U_m$ packs by translations. In particular, there is a subset $U^\prime \subset U$ which packs by translations such that $\left |U^\prime \right| = \frac 1M \left|U\right |$.
\end{proposition}

Dilations $A$ for which measurable tilings exist were characterized by Larson, Schulz, Taylor, and the second author \cite{LarSchSpeTay06}.

\begin{theorem}\label{LSST}
Let $A$ be an invertible matrix. 
\begin{enumerate}[(i)]
        \item There exists a set that tiles by dilations if and only if $A$ is not 
orthogonal.
        \item There exists a set of finite measure that tiles by dilations if and 
only if $\left|\det A\right| \not= 1$.
        \item There exists a bounded set that tiles by dilations if and only if 
all (real or complex) eigenvalues of $A$ or $A^{-1}$ have modulus larger than 1. 
\end{enumerate}
\end{theorem}

Given one measurable tiling by dilations (or translations), all such sets that tile can be constructed in the following manner. 
Let $U$ be a set that tiles by $A$ dilations. Let $\{U_j: j\in \Z\}$ be a measurable partition of $U$. Then, $\bigcup_{j\in \Z} A^j(U_j)$ tiles by $A$ dilations. Moreover, given a set $T$ that tiles by $A$ dilations, define $T_j = A^j(U) \cap T$ and $U_j = A^{-j}(T_j)$. Then $\{U_j: j\in \Z\}$ is a measurable partition of $U$ such that $\bigcup_{j\in \Z} A^j(U_j) = T.$  Similar results hold for sets that tile by translations.

Given a set $W$ that packs by $A$ dilations, we define the dilation equivalency mapping $d=d_W$ onto $W$ by 
\begin{equation}\label{dem}
d(V) = \bigcup_{j\in\Z} \left(A^j(V) \cap W\right), \qquad\text{where }V \subset \R^n.
\end{equation}
Then we have the following useful result.

\begin{proposition}\label{convergeD} 
Let $d$ be the equivalency mapping of a set $W$ that packs by dilations. Suppose that $(U_k)_{k\in \N}$ is a sequence of sets, which pack by dilations and by translations, and converges in the symmetric difference metric to $U$.  The following holds: 
\begin{enumerate}[(i)]
\item
The set $U$ packs by dilations and by translations. 
\item If
\begin{equation}\label{cd1}
\sum_k |d(U_k \triangle U_{k + 1})| < \infty,
\end{equation}
then $d(U_k) \to d(U)$ in the symmetric difference metric as $k\to \infty$. 
\item
If $V \subset W$ is such that
\begin{equation}\label{cd2}
\sum_k |d(U_k \triangle U_{k + 1}) \cap V| < \infty,
\end{equation}
then $d(U_k) \cap V \to d(U) \cap V$ in the symmetric difference metric as $k\to \infty$. 
\end{enumerate}
\end{proposition}

\begin{proof}
Condition (i) is an exercise and can be found in \cite[Lemma 3.1]{Spe03}. 
To prove (ii), note that 
\[
U\triangle U_k \subset \bigcup_{j = k}^\infty \bigl(U_{j+ 1} \triangle U_j\bigr).
\]
By \eqref{cd1} it follows that $d(U \triangle U_k) \to \emptyset$ as $k \to \infty$. Since $d(U) \triangle d(U_k) \subset d(U \triangle U_k)$, the conclusion (ii) follows.

Condition (iii) follows by noting that $d_V(U) = d_W(U) \cap V$ and applying (ii).
 \end{proof}
 
 Finally, we recall a standard fact on wavelet sets \cite[Theorem 2.2]{IonWan06}, which is a consequence of the Cantor-Schr\"oder-Bernstein theorem.
To show the existence of a wavelet set, it suffices to construct a set which tiles by dilations and packs by translations.

\begin{theorem}\label{csb} Let $A$ be an invertible matrix and let $\Gamma$ be a full rank lattice in $\R^n$.  Suppose there exists a measurable set  $U \subset \R^n$ that tiles by $A$ dilations and packs by $\Gamma$ translations. Then, there exists $(A,\Gamma)$ wavelet set.
\end{theorem}

\section{Proof of sufficiency for existence of wavelet sets} \label{sufficiency}

The goal of this section is to prove the sufficiency part of the main theorem.

\begin{theorem}\label{mainTheorem} Let $A$ be an invertible $n \times n$ matrix  and let $\Gamma$ be a full rank lattice. If for some $r>0$
\begin{equation}\label{suf0}
\left | \det A \right | > 1 \qquad\text{and}\qquad
\sum_{j=1}^\infty \frac{1}{\#\left |A^{-j}(\mathbf B(0,r)) \cap \Gamma \right|} =\infty,
\end{equation}
or 
\begin{equation}\label{suff}
\left | \det A \right | < 1 \qquad\text{and}\qquad \sum_{j=1}^\infty \frac{1}{\#\left | A^{j}(\mathbf B(0,r)) \cap \Gamma \right|} =\infty,
\end{equation}
then there exists an $(A, \Gamma)$ wavelet set.
\end{theorem}

Note that the second part of Theorem \ref{mainTheorem} follows from the first part by replacing $A$ with $A^{-1}$. Hence, from now on we shall assume that $\left | \det A \right |>1$ and we prove  the existence of a wavelet set under the assumption \eqref{suf0}. We start with an elementary lemma.

\begin{lemma}\label{as}
Let $A$ be an invertible matrix. Let $\Gamma$ be a full rank lattice in $\R^n$.
The following are equivalent:
\begin{enumerate}[(i)]
\item  for every $r > 0$, there exists a sequence $(m_j)_{j\in \N}$ such that $\sum 1/m_j=\infty$ and a set $A^{-j}(\mathbf B(0,r))$ packs $m_j$ redundantly via $\Gamma$ translations for any $j\in\N$,
\item $\sum_{j=1}^\infty 1/ \# | A^{-j} (\mathbf B(0,r)) \cap \Gamma|  =\infty$ for every $r > 0$,
\item $\sum_{j=1}^\infty 1/ \# | A^{-j} (\mathbf B(0,r)) \cap \Gamma|  =\infty$ for some $r > 0$.
\end{enumerate}
\end{lemma}

\begin{proof}
The condition (i) means that for any $r>0$ there exists a sequence $(m_j)$ such that $\sum 1/m_j=\infty$ and
\[
\sum_{\gamma \in \Gamma} \ch_{A^{-j} (\mathbf B(0,2r))}(x+\gamma) \le m_j \qquad\text{for a.e. }x\in \R^n, \ j\in\N.
\]
Since 
\[
A^{-j}(\bfb(0, r)) \subset A^{-j}(\bfb(0, 2r))+ \gamma\qquad\text{for all }\gamma \in A^{-j}(\bfb(0, r)) \cap \Gamma,
\]
it follows that $\#\left|A^{-j}(\bfb(0, r)) \cap \Gamma\right| \le m_j$ and 
\[
\sum \frac {1}{\#\left|A^{-j}(\bfb(0, r)) \cap \Gamma\right|} \ge \sum 1/m_j  = \infty.
\]

The implication (ii) $\implies$ (iii) is trivial. 
Finally, assume that (iii) holds for some $r_0>0$. Observe that
\begin{equation}\label{as5}
\sum_{\gamma \in \Gamma} \ch_{A^{-j} (\mathbf B(0,r_0/2))} (x+\gamma) \le \# | A^{-j} (\mathbf B(0,r_0)) \cap \Gamma|  \qquad\text{for all }x\in \R^n.
\end{equation}
Indeed, if there exist $k$ distinct elements $\gamma_1,\ldots,\gamma_k \in \Gamma$, such that $x+\gamma_1,\ldots, x+\gamma_k \in A^{-j}(\mathbf  B(0,r_0/2))$, then we have $k$ distinct elements $\gamma_i-\gamma_1$, $i=1,\ldots,k$, belonging to $A^{-j}(\mathbf B(0,r_0))$.  Take any $r>0$. Then, there exists $M>0$ such that a ball of radius $r$ can be covered by a union of $M$ balls of radius $r_0/2$. Hence, \eqref{as5} yields
\[
\sum_{\gamma \in \Gamma} \ch_{A^{-j} (\mathbf B(0,r))}(x+\gamma) \le M  \# | A^{-j} (\mathbf B(0,r_0)) \cap \Gamma|  \qquad\text{for all }x\in \R^n.
\]
Letting $m_j =M  \# | A^{-j} (\mathbf B(0,r_0)) \cap \Gamma|$, $j\in\N$, implies (i) and completes the proof of the lemma.
\end{proof}

We also need a simple lemma about sequences.

\begin{lemma}\label{seq}
Let $(a_i)_{i\in \N}$ be a sequence of numbers in $[0,1]$. Let $(b_i)_{i\in \N}$ be a sequence defined recursively by 
$b_1=1$, $b_{i+1}=(1-a_{i+1}) b_i$, $i\in \N$. Then, $\sum a_{i+1} b_{i} <\infty$.
\end{lemma}

\begin{proof}
Define a sequence $(c_i)_{i\in \N}$ by 
\[
c_1=1, \qquad c_{i}=\exp\bigg(-\sum_{j=2}^{i} a_i \bigg), \ i\ge 2.
\] 
Using $1-x \le \exp(-x)$ for all $x\in \R$ we have $b_i \le c_i$ for all $i\in \N$. Applying the Mean Value Theorem for $\exp(-x)$ yields 
\[
 a_{i+1}c_{i+1}
=  a_{i+1} \exp\bigg(-\sum_{j=2}^{i+1} a_i \bigg)
 \le \exp\bigg(-\sum_{j=2}^{i} a_i \bigg) - \exp\bigg(-\sum_{j=2}^{i+1} a_i \bigg) .
\]
Telescoping yields $\sum a_{i+1}c_{i+1} <\infty$. Since $c_i \le e c_{i+1}$, we deduce that $\sum a_{i+1} b_i \le \sum a_{i+1} c_{i} <\infty$.
\end{proof}

Given a set $U$ that packs by $\Gamma$ translations, we define the translation equivalency mapping $\tau_U$ onto $U$ by 
\begin{equation}\label{tau}
\tau_U(V) = \bigcup_{\gamma \in \Gamma} (\gamma + V) \cap U, \qquad\text{where }V \subset \R^n.
\end{equation}
The following lemma plays the key role in the proof of Theorem \ref{mainTheorem}. 

\begin{lemma}\label{easyCase} 
Let $A$ be an invertible $n \times n$ matrix and $\Gamma$ a full-rank lattice.
Let $W$ be a set that packs by dilations of $A$ and let $U\subset W$ be a finite measure subset. 
Suppose there exists a sequence $(m_j)_{j\in\N}$ of positive integers such that $\sum_{j\in\N} 1/m_j=\infty$ and $A^{-j}(U)$ packs $m_j$ redundantly via translations for every $j\in\N$.  Then, for every $\epsilon > 0$, there exists a set $V$ with the following properties:
\begin{enumerate}[(i)]
\item $V$ packs via dilations,
\item $V$ packs via translations,
\item $d(V) = U$, where $d$ is the dilation equivalency mapping onto $W$ given by \eqref{dem},
\item $|V| < \epsilon$.
\end{enumerate}
\end{lemma}

\begin{proof}
Let $J\in \N$ be such that
\[
\sum_{j=J}^\infty | A^{-j}(U) | < \epsilon.
\]
Since the set $V$ is going to be chosen as a subset of $\bigcup_{j\ge J} A^{-j} (U)$,  property (iv) will be satisfied by $V$. 


We shall construct inductively a sequence of sets $(U_j)_{j\ge J}$ such that for all $j\ge J$ we have (letting $U_{J - 1} = \emptyset$): 
\begin{enumerate}[(a)]
\item $U_j$ packs by translations and dilations,
\item letting $W_j:=U \setminus d(U_j)$ we have $|W_{j}| \le (1-c/m_j)|W_{j-1}|$, where $c=1-\left | \det A \right |^{-1}$,
\item $U_{j} \setminus U_{j-1} \subset A^{-j}(W_{j-1})$,
\item $U_{j-1} \setminus U_{j} \subset \tau_{U_{j-1}}(A^{-j}(W_{j-1}))$.
\end{enumerate}

 By Lemma \ref{easyProp} there is a set $U_J \subset A^{-J}(U)$ which packs by translations and $|U_J| \ge \frac 1{m_J} |A^{-J}(U)|$. Let $W_J = U \setminus d(U_J)$. Then, $|W_J| = |U \setminus d(U_J)| \le \frac {m_J - 1}{m_J} |U| \le \left(1 - c/m_J\right) |U|$, as desired. 

Suppose that $(U_j)_{j = J}^k$ and $(W_j)_{j = J}^k$ have been defined so that items (a)--(d) above hold. 
Since $A^{-j-1} U$ packs $m_{j+1}$ redundantly by translations and $W_j \subset U$, there exists a set $\tilde U_{j + 1} \subset A^{-j-1} W_j$, which packs by translations (and dilations) such that
\begin{equation}\label{wm0}
|\tilde U_{j + 1}| = \frac 1{m_{j+1}} |A^{-j-1} W_j|.
\end{equation}
Let $U_{j + 1} = \tilde U_{j + 1} \cup \left(U_j \setminus \tau_{U_j} (\tilde U_{j + 1})\right)$. We see from construction that $U_{j + 1}$ satisfies (a), (c), and (d). It remains to see that $U_{j + 1}$ satisfies (b). 

By (c) we have
\begin{equation}\label{wm1}
U_j \subset \bigcup_{l=J}^j A^{-l}(U).
\end{equation}
Since 
$\tilde U_{j+1} \subset A^{-j-1}(U)$, by \eqref{wm0} we have 
\begin{equation}\label{wm3}
|d(\tilde U_{j+1})|= \left | \det A \right |^{j+1}|\tilde U_{j+1}|= \frac {|W_j|}{m_{j+1}}.
\end{equation}
Hence, \eqref{wm1} and
$
|\tau_{U_j}(\tilde U_{j+1})| = |\tilde U_{j+1}|$ yields
\begin{equation}\label{wm2}
|d(\tau_{U_j}(\tilde U_{j+1}))|  \le 
 \left | \det A \right |^{j} |\tilde U_{j+1}| = \frac{|d(\tilde U_{j+1})|}{\left | \det A \right |}.
\end{equation}
By \eqref{wm2} and noting that $d(\tilde U_{j + 1})$ and $d(U_j)$ are disjoint, we have
\begin{align*}
    |W_{j + 1}| &= |U \setminus d(U_{j + 1})|\\
    &\le |U| - \bigg(|d(\tilde U_{j + 1})| + |d(U_j)| - \frac{|d(\tilde U_{j+1})|}{\left | \det A \right |} \bigg) 
    \\
    & = |W_j| - c |d(\tilde U_{j + 1})| 
    = |W_j| - c \frac{|W_j|}{m_{j+1}} = \bigg(1- \frac{c}{m_{j+1}}\bigg) |W_j|,
\end{align*}
where $c=1-\left | \det A \right |^{-1}$. This proves (b).

By (c) and (d) the constructed sets $(U_j)_{j \ge J}$ satisfy
\[
U_{j+1} \triangle U_{j} = (U_{j+1} \setminus U_{j}) \cup (U_{j} \setminus U_{j+1}) \subset 
\tilde U_{j+1} \cup  \tau_{U_j}(\tilde U_{j+1}) 
\subset
A^{-j-1}(W_{j}) \cup \tau_{U_j}(A^{-j-1}(W_j)),
\]
where $\triangle$ is the symmetric difference of the sets.
Thus,
\begin{equation}\label{p5}
\sum_{j=J}^\infty |U_{j+1} \triangle U_j| \le \sum_{j=J}^\infty 2 \left | \det A \right |^{-j-1} |W_{j}| < \infty.
\end{equation}
Likewise, by \eqref{wm3} and \eqref{wm2} we have
\[
|d(U_{j+1} \triangle U_{j})| \le |d(\tilde U_{j+1})| +|d(  \tau_{U_j}(\tilde U_{j+1})) |
\le 2 |d(\tilde U_{j+1})| = \frac{2}{m_{j+1}} |W_j|.
\]
Consequently, 
 by property (b) and Lemma \ref{seq} we have
\begin{equation}\label{p6}
\sum_{j=J}^\infty |d(U_{j+1} \triangle U_j)| \le 2 \sum_{j=J}^\infty \frac{|W_j|}{m_{j+1}} < \infty.
\end{equation}

By Proposition \ref{convergeD}, the sets $U_j$ converge in the symmetric difference metric to a set $V$ that packs by translations and by dilations, and $d(V) = \lim_{j \to \infty} d(U_j) = U$. Indeed, the last equality is a consequence of property (b) and the assumption $\sum 1/m_j=\infty$ since
\[
|U \setminus d(U_{j})| = |W_{j}| \le \prod_{i=J+1}^j \bigg(1- \frac c{m_i} \bigg ) |W_{J}| \le \exp\bigg(- \sum_{i=J+1}^j \frac c{m_i}\bigg) |W_J| \to 0 \qquad\text{as }j\to \infty.
\] This proves the lemma. 
\end{proof}

We are now ready to prove the main result that we use to deduce Theorem \ref{mainTheorem}. 

\begin{theorem}\label{weakLemvig} Let $A$ be an invertible $n \times n$ matrix and $\Gamma$ a full-rank lattice. Suppose there exists a set $W\subset \R^n$ which tiles $\R^n$ by dilations and a partition $(W_m)_{m\in \N}$ of $W$ such that for each $m\in \N$, there exists a sequence $(m_j)_{j\in\N}$ of positive integers such that $\sum_{j\in\N} 1/m_j=\infty$ and $A^{-j}(W_m)$ packs $m_j$ redundantly via translations for every $j\in\N$. 
Then, there exists an $(A, \Gamma)$ wavelet set.
\end{theorem}

\begin{proof}
Without loss of generality, we may assume that each set in the partition of $W$ has measure less than 1/2. Let $d$ be the dilation equivalency mapping onto $W$.
We shall construct inductively two sequences of sets $(U_k)_{k\in\N}$ and $(\tilde U_k)_{k\in\N}$ such that for all $k\in\N$ we have: 
\begin{enumerate}[(a)]
\item $U_k$ packs by translations and dilations,
\item $\tilde U_k$ packs by translations and dilations,
\item $U_{k+1}=\tilde U_k \cup (U_k \setminus \tau_{U_k}(\tilde U_k))$,
\item $d(\tilde U_k)= \bigl(\bigcup_{j=1}^{k+1} W_j\bigr) \setminus d(U_k)$,
\item $|d(\tau_{U_k}(\tilde U_k))|< 1/2^k$.
\end{enumerate}

By Lemma \ref{easyCase} there exists $U_1$ such that $U_1$ packs by translations and dilations and $d(U_1) = W_1$. Likewise, by Lemma \ref{easyCase} there exists $\tilde U_1$ such that $\tilde U_1$ packs by translations and dilations, $d(\tilde U_1) = W_2$, and $|\tilde U_1|<1/2$.  Then, $ |d(\tau_{U_1}(\tilde U_1))| \le |d(U_1)|=|W_1|< 1/2$. Consequently, (a)--(e) hold for $k=1$.

For the inductive step, suppose that we have already constructed sets $U_1,\ldots,U_k$ and $\tilde U_1,\ldots, \tilde U_k$ satisfying (a)--(e). Then, the set $U_{k+1}$ is determined by the property (c). Since both $U_k$ and $\tilde U_k$ pack by translations, so does $U_{k+1}$. To see that $U_{k+1}$ packs by dilations, note that both $U_{k}$ and $\tilde U_k$ pack by dilations, and by property (d)
\[
d(\tilde U_{k}) \cap d( U_k \setminus \tau_{U_k}(\tilde U_{k})) \subset \bigg(\bigcup_{j=1}^{k+1} W_j \setminus d(U_k) \bigg) \cap d(U_k) = \emptyset.
\]
This proves (a) for $k+1$.

Since $d(U_{k+1}) \subset W_1 \cup \ldots \cup W_{k+1}$ has finite measure, there exists $K \in \N$ such that 
\begin{equation}\label{q2}
\bigg|d \bigg(U_{k+1} \cap \bigcup_{j \ge K} A^{-j}(W) \bigg)\bigg| < 1/2^{k+2}.
\end{equation}
By Lemma \ref{easyCase} there exists $\tilde U_{k+1}$ such that $\tilde U_{k+1}$ packs by translations and by dilations,
\[
d(\tilde U_{k+1}) = \bigcup_{j=1}^{k+2} W_j \setminus d(U_{k+1}).
\]
and
\begin{equation}\label{q4}
|\tilde U_{k+1}| < \epsilon: =  \left | \det A \right |^{-K}/2^{k+2}.
 \end{equation}
This guarantees that (b) and (d) hold for $k+1$. Hence, it remains to show the estimate (e) for $k+1$.
By \eqref{q2}, \eqref{q4}, and the fact that $|\tau_{U_{k+1}}(\tilde U_{k+1})|<\epsilon$, we have
\begin{align*}
 |d(\tau_{U_{k+1}}(\tilde U_{k+1}))| 
 &\le \bigg|d \bigg( \tau_{U_{k+1}}(\tilde U_{k+1}) \cap \bigcup_{j \ge K} A^{-j} W \bigg) \bigg| 
    	+ \bigg|d \bigg (\tau_{U_{k+1}}(\tilde U_{k+1}) \cap \bigcup_{j < K} A^{-j} W \bigg) \bigg| \\
    &\le 1/2^{k+2} + \left | \det A \right |^{K} |\tau_{U_{k+1}}(\tilde U_{k+1})| \le 1/2^{k+1}.
\end{align*}
Indeed, the second inequality is a consequence of a fact that for any measurable set $V \subset \R^n$ we have
\[
 \bigg|d \bigg(V \cap \bigcup_{j < K} A^{-j} W \bigg) \bigg| \le  \left | \det A \right |^{K} |V|.
 \]
 This completes the proof of properties (a)--(e).
 
By (a), (c), and (d) we have
\begin{equation}\label{q6}
\begin{aligned}
d(U_{k+1}) =  d(\tilde U_k) \cup d(U_k \setminus \tau_{U_k}(\tilde U_k))
&=\bigg(\bigcup_{j=1}^{k+1} W_j \setminus d(U_k)\bigg) \cup 
( d(U_k) \setminus d(\tau_{U_k}(\tilde U_k)))
\\
& =\bigcup_{j=1}^{k+1} W_j \setminus d(\tau_{U_k}(\tilde U_k)).
\end{aligned}
\end{equation}
Hence, for $k\ge 2$ we have
\[
\begin{aligned}
d(U_{k+1} \setminus U_{k}) \subset d(\tilde U_k)
&= \bigcup_{j=1}^{k+1} W_j \setminus d(U_k)
= \bigcup_{j=1}^{k+1} W_j \setminus \bigg(\bigcup_{j=1}^{k} W_j \setminus d(\tau_{U_{k-1}}(\tilde U_{k-1}))\bigg)
\\
& \subset W_{k+1} \cup d(\tau_{U_{k-1}}(\tilde U_{k-1})).
\end{aligned}
\]
On the other hand,
\[
U_{k} \setminus U_{k+1} \subset U_k \setminus (U_k \setminus \tau_{U_{k}}(\tilde U_{k})) =\tau_{U_{k}}(\tilde U_{k}).
\]
Combining the last two inclusions yields
\begin{equation}\label{q8}
d(U_{k+1} \triangle U_{k}) \subset W_{k+1} \cup d(\tau_{U_{k}}(\tilde U_{k})) \cup d(\tau_{U_{k-1}}(\tilde U_{k-1})).
\end{equation}
Likewise, we have
\begin{equation}\label{q10}
U_{k+1} \triangle U_{k} \subset \tilde U_{k} \cup \tau_{U_k}(\tilde U_k).
\end{equation}
Hence, by  \eqref{q4} and \eqref{q10}
\begin{equation}\label{q12}
\sum_{k\in\N} | U_{k+1} \triangle U_{k} | \le 2 \sum_{k\in \N} |\tilde U_k| <\infty.
\end{equation}

Define
\[
U = \bigcap_{j = 1}^\infty \bigcup_{k = j}^\infty U_k.
\]
By \eqref{q12} $U_k \to U$ in the symmetric difference metric, from which it follows that $U$  packs by dilations and translations. We claim that $U$ tiles by dilations. 

Fix $K\in\N $. By \eqref{q8} and property (e) we have
\[
\sum_{k=K}^\infty \bigg |d(U_{k+1} \triangle U_k) \cap \bigcup_{j=1}^K
W_j \bigg| \le \sum_{k=K}^\infty |d(\tau_{U_{k}}(\tilde U_{k}))|+ |d(\tau_{U_{k-1}}(\tilde U_{k-1}))|< \infty
\]
Thus, by Proposition \ref{convergeD}(iii) we have
\[
d(U_k) \cap \bigcup_{j=1}^K
W_j \to d(U) \cap \bigcup_{j=1}^K W_j
\qquad\text{as } k\to \infty.
\]
On other hand, by \eqref{q6} and property (e)
\[
d(U_k) \cap \bigcup_{j=1}^K W_j \to \bigcup_{j=1}^K W_j \qquad\text{as }k\to\infty.
\]
Since $K\in \N$ is arbitrary, we conclude that $d(U) = \bigcup_{j=1}^\infty W_j= W$. Hence, $U$ tiles by dilations and packs by translations. Theorem \ref{csb} implies that there exists a set that tiles both by dilations and translations.
\end{proof}

Finally we can complete the proof of Theorem \ref{mainTheorem}.

\begin{proof}[Proof of Theorem \ref{mainTheorem}]
Without loss of generality we can assume that $\left | \det A \right | > 1$. By Lemma \ref{as}, the condition \eqref{suf0} implies that there exists
a sequence $(m_j)_{j\in\N}$  such that $\sum 1/m_j =\infty$ and
  $A^{-j}(\mathbf B(0,r))$ packs $m_j$ redundantly by translations for all $j\in \N$. Take any set $W\subset \R^n$ that tiles by dilations. The partition $W_m=W \cap \bigl(\mathbf B(0,m) \setminus \mathbf B(0,m-1)\bigr)$, $m\in \N$, fulfills assumptions of Theorem \ref{weakLemvig}. Consequently, there exists an $(A, \Gamma)$ wavelet set.
\end{proof}

The contrapositive of Theorem \ref{mainTheorem} is as follows. 

\begin{corollary}\label{cin}
If a pair $(A,\Gamma)$ does not admit a wavelet set and $\left | \det A \right |>1$, then
\begin{equation}\label{cin0}
\sum_{j=1}^\infty \frac1{\#|A^{-j}(\mathbf B(0,1))\cap \Gamma|} <\infty.
\end{equation}
\end{corollary}

\section{Proof of necessity for existence of wavelet sets} \label{necessity}

The goal of this section is to prove the converse of Corollary \ref{cin}. That is, we prove the necessity part of the main theorem.

\begin{theorem} \label{nece}
Let $A$ be an $n \times n$ invertible matrix with $\left | \det A \right | > 1$. Let $\Gamma \subset \R^n$ be a full rank lattice. If \eqref{cin0} holds, 
then there is no $(A,\Gamma)$ wavelet set.
\end{theorem}

We need the following two elementary lemmas, which may be of independent interest.

\begin{lemma}\label{elp}
Let $\mathcal E \subset \R^n$ be an ellipsoid centered at $0$. Let $V \subset \R^n$ be a subspace of dimension $d$. Let $Q$ be the orthogonal projection of $\R^n$ onto 
\[
V^\perp=\{ x\in \R^n: \langle x, y \rangle =0 \quad\text{for all }y\in V \}.
\]
Then,
\[
m_d( \mathcal E \cap V) m_{n-d}( Q(\mathcal E)) \le 2^n m_n(\mathcal E).
\]
\end{lemma}

\begin{proof}
By Fubini's theorem
\begin{equation}\label{elp2}
m_n(\mathcal E) = \int_{V^\perp} m_d(\mathcal E \cap (x+V)) dm_{n-d}(x) \ge 
\int_{Q(\frac12 \mathcal E)} m_d(\mathcal E \cap (x+V)) dm_{n-d}(x) .
\end{equation}
Take any $x_0\in Q(\frac12 \mathcal E)$. Let $x_1 \in \mathcal E$ be such that $x_0=\frac12 Q(x_1)$. Consider a mapping $F:  V \to x_0+ V$ given by $F(x) = (x+x_1)/2$.  Since $\mathcal E$ is convex, we have $F(\mathcal E \cap V) \subset \mathcal E \cap (x_0+V)$. Since $F$ is affine
\begin{equation}\label{elp4}
m_d(\mathcal E \cap (x_0+V)) \ge m_{d}(F(\mathcal E \cap V)) = 2^{-d} m_d(\mathcal E \cap V).
\end{equation}
Combining \eqref{elp2} and \eqref{elp4} yields
\[
m_n(\mathcal E) \ge 2^{-d} m_d(\mathcal E \cap V) m_{n-d}(\tfrac12 Q(\mathcal E))
= 2^{-n} m_d(\mathcal E \cap V) m_{n-d}(Q(\mathcal E)).
\qedhere
\]
\end{proof}

\begin{lemma}\label{scz}
Let $\Gamma \subset \R^n$ be a full rank lattice in $\R^n$. Suppose that a measurable set $W\subset \R^n$ packs by $\Gamma$ translations,
\begin{equation}\label{sca}
\sum_{\gamma \in \Gamma} \ch_{W}(x+\gamma) \le 1
\qquad\text{for a.e. }x\in \R^n.
\end{equation}
Let $V$ be a $d$-dimensional lattice subspace of $\R^n$. That is, $\dim V=d$ and $\Gamma \cap V$ is a lattice of rank $d$.
Let $F_0$ be a fundamental domain of $\Gamma \cap V$ in $V$. 
Then, we have
\begin{equation}\label{scb}
m_d(W \cap (y+V)) \le m_d(F_0) \qquad\text{$m_{n-d}$-a.e. } y\in V^\perp.
\end{equation}
\end{lemma}

\begin{proof}
Let $K \subset V^\perp$ be a measurable set. By \eqref{sca} we have
\[
\int_{K+F_0} \sum_{\gamma \in \Gamma \cap V} \ch_W(x+\gamma) dm_n(x) \le \int_{K+F_0} dm_n(x) = m_{n-d}(K)m_{d}(F_0).
\]
By Fubini's Theorem
\[
\begin{aligned}
\int_{K+F_0} \sum_{\gamma \in \Gamma \cap V} \ch_W(x+\gamma) dm_n(x) 
& = \int_K \int_{F_0} \sum_{\gamma \in \Gamma \cap V} \ch_W(x_1+x_2+\gamma) dm_d(x_1)dm_{n-d}(x_2)
\\
&= \int_K m_d((x_2+V) \cap W) dm_{n-d}(x_2).
\end{aligned}
\]
Hence,
\[
\int_K m_d((x_2+V) \cap W) dm_{n-d}(x_2) \le m_d(F_0) \int_K dm_{n-d}(x_2).
\]
Since $K \subset V^\perp$ is arbitrary, we deduce \eqref{scb}.
\end{proof}

 We will also need a volume packing lemma which can be found in the book of Tao and Vu \cite[Lemma 3.24 and Lemma 3.26]{TaoVu06}.

\begin{lemma}\label{vp}
Let $\Gamma \subset \R^n$ be a full rank lattice, and let $\Omega$ be a symmetric convex body in $\R^n$. Then,
\begin{equation}\label{eq:vp1}
\frac{|\Omega|}{2^n |\R^n/\Gamma|} \le \#|\Omega \cap \Gamma|.
\end{equation}
In, addition if the vectors $\Omega \cap \Gamma$ linearly span $\R^n$, then
\begin{equation}\label{eq:vp2}
\#|\Omega \cap \Gamma| \le \frac{3^n n! |\Omega|}{2^n |\R^n/\Gamma|}.
\end{equation}
\end{lemma}

The following lemma plays the key role in the proof of Theorem \ref{nece}. 

\begin{lemma}\label{slice}
Let $A$ be an $n \times n$ invertible matrix and let $\Gamma \subset \R^n$ be a full rank lattice.
Suppose that a measurable set $W\subset \R^n$ packs by $\Gamma$ translations, that is, \eqref{sca} holds.
Then, there exists a constant $C>0$, which depends only on the dimension $n$, such that
\begin{equation}\label{slice0}
|\mathbf B(x,1) \cap A(W)| \le \frac{C}{\#|A^{-1}(\mathbf B(0,1))\cap \Gamma|} \qquad\text{for all } x\in\R^n.
\end{equation}
\end{lemma}

\begin{proof}
If $A^{-1}(\mathbf B(0,1))\cap \Gamma =\{0\}$, then the inequality \eqref{slice0} holds automatically with $C=m_n(\mathbf B(0,1))$. Otherwise, we define $V$ to be the linear span of $A^{-1}(\mathbf B(0,1))\cap \Gamma $. Let $F_0$ be a fundamental domain of $V/(\Gamma \cap V)$  and $d=\dim V$.
By Lemma \ref{scz}
\begin{equation}\label{sc1}
m_d(W \cap (y+V)) \le m_d(F_0) \qquad\text{for a.e. }y\in V^\perp.
\end{equation}
Fubini's theorem yields
\begin{equation}\label{sc3}
\begin{aligned}
|A^{-1}(\mathbf B(x,1)) \cap W| &\le \int_{V^\perp} m_d(A^{-1}(\mathbf B(x,1)) \cap W \cap (y+V)) dm_{n-d}(y)
\\
&\le
\int_{Q(A^{-1}(\bfb(x, 1)))} m_d(W \cap (y + V)) \, dm_{n - d}(y)
\\
&\le
m_d(F_0) m_{n-d}\left(Q(A^{-1} (\mathbf B(x,1)))\right),
\end{aligned}
\end{equation}
where $Q$ is an orthogonal projection of $\R^n$ onto $V^\perp$. The second inequality is because the integrand is nonzero only if $A^{-1}(\mathbf B(x,1))  \cap (y+V) \not= \emptyset$, which implies that $y\in Q\left(A^{-1} (\mathbf B(x,1))\right)$.
Therefore, by Lemmas \ref{elp} and \ref{vp}
\begin{equation}\label{sc6}
\begin{aligned}
|\mathbf B(x,1)\cap A(W)| & \le \left | \det A \right | m_d(F_0) m_{n-d}\left(Q(A^{-1} (\mathbf B(0,1)))\right)
\le 
\frac{2^n m_n(\mathbf B(0,1)) m_d(F_0)}{ m_{d}(A^{-1} (\mathbf B(0,1)) \cap V)} 
\\
&\le \frac{3^d d!   }{2^{d}}  \frac{2^n m_n(\mathbf B(0,1))}{\#|A^{-1} (\mathbf B(0,1)) \cap \Gamma|},
\end{aligned}
\end{equation}
Hence, the inequality \eqref{slice0} holds with $C=3^n n! m_n(\mathbf B(0,1))$.
\end{proof}

\begin{lemma}\label{dti}
Let $A$ be $n\times n $ matrix such that $\left | \det A \right |>1$. Then, there exists a set $U \subset \R^n$ such that
 $\{A^j(U): j\in \Z\}$ is a tiling of $\R^n$ modulo null sets and $U$
 contains a collection of disjoint balls $V_k \subset U$, $k\in \N$, each having the same radius. 
\end{lemma}

\begin{proof}
Since $\left | \det A \right |>1$, the matrix $A$ has at least one eigenvalue $\lambda$ such that $|\lambda|>1$. Suppose momentarily that $\lambda$ is real. For simplicity assume that $e_1$ is an eigenvector of $A$ with eigenvalue $\lambda$. Then,
\begin{equation}\label{ug}
U= ((-\lambda,-1) \cup (1,\lambda)) \times \R^{n-1}
\end{equation}
is a dilation tiling generator of $\R^n$. Likewise, if $\lambda$ is complex, then a real Jordan normal form of $A$ contains a $2\times 2$ block $M_\lambda$, see the proof of Lemma \ref{ite}. For simplicity assume that $A$ is already in a Jordan normal form with $M_\lambda$ as its first block. Then, 
\begin{equation}\label{ug2}
U= \{x\in \R^2: 1<|x|<|\lambda| \} \times \R^{n-2}
\end{equation}
is a dilation tiling generator of $\R^n$. That is, $\{A^j(U): j\in \Z\}$ is a tiling of $\R^n$ modulo null sets. As a consequence of \eqref{ug} and \eqref{ug2} we deduce that there exists a collection of disjoint balls $V_k \subset U$, $k\in \N$, each having the same radius that does not depend on $k$. The simplifying assumptions, which we made on $A$, do not affect this conclusion.
\end{proof}

The general principle that we will use to prove the non-existence of wavelet sets is the following proposition, which was implicit in earlier work \cite{Spe03}.

\begin{proposition}\label{generalprinciple}
Let $A$ be an invertible $n \times n$ matrix with $\left | \det A \right | > 1$, let $\Gamma \subset \R^n$ be a full-rank lattice, and let $W \subset \R^n$. Suppose that there exists $\epsilon > 0$ and an $A$ dilation generator $U$ which contains infinitely many disjoint sets $(V_k)_{k = 1}^\infty$ of measure at least $\epsilon$ such that 
\[
a_J:= \sup_{k \in \N} \sum_{j = J}^\infty \bigl |V_k \cap A^j(W) \bigr | \to 0 \qquad\text{as } J \to \infty.
\]
Then, $W$ is not an $(A, \Gamma)$ wavelet set.
\end{proposition}

\begin{proof}
Suppose for the sake of contradiction that $W$ is an $(A, \Gamma)$ wavelet set. 
For $j\in \Z$ we define sets $W_j = A^{-j}(U) \cap W$. The collection $\{W_j:j\in \Z\}$ is a partition of $W$. 
For any $j\in \Z$ we have that 
\[
|W_j|=|A^{-j}(U) \cap W| \le |F|,
\]
where $F \subset \R^n$ is a fundamental domain of $\Gamma$. Hence, for $J \in \N$ chosen so that $a_J < \epsilon/ 2$, we have
\begin{equation}\label{finite1}
\sum_{j=-\infty}^J | A^{j}(W_{j})| \le \sum_{j=-\infty}^J \left | \det A \right |^{j}|F| = \frac{\left | \det A \right |^{J} |F|}{1-\left | \det A \right |^{-1}} < \infty.
\end{equation}
Since the sets $V_k$, $k\in \N$, are disjoint, by \eqref{finite1} there exists $k\in\N$ such that 
\[
\bigg|V_k \cap \bigcup_{j=-\infty}^J  A^{j}(W_{j})\bigg| < |V_k|/2.
\]
Additionally, by the definition of $a_J$ we have that 
\[
\epsilon > \sum_{j = J}^\infty |V_k \cap A^j(W)| \ge \bigg|\bigcup_{j = J}^\infty V_k \cap A^j(W)\bigg| = \biggl|V_k \cap \bigcup_{j = J}^\infty A^j(W_j) \bigg|.
\]
Therefore, $|V_k \cap \bigcup_{j \in \Z} A^j(W_j)| < |V_k|/2 + \epsilon/2 < |V_k|$, and $W$ is not an $A$ dilation generator of $\R$. In particular, $W$ is not an $(A, \Gamma)$ wavelet set.
\end{proof}

We are now ready to prove Theorem \ref{nece}.

\begin{proof}[Proof of Theorem \ref{nece}]
Suppose that $W\subset \R^n$ is any measurable set which tiles by $\Gamma$ translations. We shall use Proposition \ref{generalprinciple} to deduce that $W$ can not be an $(A,\Gamma)$ wavelet set.
By Lemma \ref{dti} there exists an $A$ dilation tiling generator $U \subset \R^n$ which contains a collection of disjoint balls $V_k \subset U$, $k\in \N$, each having the same radius $\le 1$. Let $x_k$ denote the center of $V_k$. Then, for any $j\in\N$, Lemma \ref{slice} applied to the matrix $A^j$ yields
\[
|V_k \cap A^j(W)| \le |\mathbf B(x_k,1) \cap A^j(W)| \le \frac{C}{\#|A^{-j}(\mathbf B(0,1))\cap \Gamma|} .
\]
By our hypothesis \eqref{cin0} we have
\[
\sum_{j = J}^\infty \bigl |V_k \cap A^j(W) \bigr | \le C \sum_{j=J}^\infty \frac{1}{\#|A^{-j}(\mathbf B(0,1))\cap \Gamma|} \to 0 \qquad\text{as }J\to \infty.
\]
Since $k\in\N$ is arbitrary, Proposition \ref{generalprinciple} implies that $W$ is not  an $(A, \Gamma)$ wavelet set.
\end{proof}

\section{Necessary condition for existence of wavelet sets} \label{oldnecessity}

In this section we show a convenient necessary condition for existence of wavelet sets, which is slightly weaker than the characterization equation \eqref{char}, but is easier to check when it does hold.
This necessary condition states that for every sublattice $\Lambda \subset \Gamma$ we have
\begin{equation}\label{nsec}
\liminf_{j\to \infty} m_d(V \cap A^{-j}(\mathbf B(0, 1))) <\infty,
\qquad\text{where }V= \spa \Lambda, \ d=\dim V.
\end{equation}
Here, $m_d$ denotes $d$-dimensional Lebesgue measure on a subspace $V\subset \R^n$.  Equivalently, we have the following non-existence result.

\begin{theorem}\label{nexist}
Let $A$ be an $n \times n$ invertible matrix with $\left | \det A \right | > 1$. Let $\Gamma \subset \R^n$ be a full rank lattice. If for any sublattice $\Lambda \subset \Gamma$ we have
\begin{equation}\label{sec}
\lim_{j\to \infty} m_d(V \cap A^{-j}(\mathbf B(0, 1)))= \infty,
\qquad\text{where }V= \spa \Lambda, \ d=\dim V,
\end{equation}
then there is no $(A,\Gamma)$ wavelet set.
\end{theorem}

As an illustration of Theorem \ref{nexist} we recover the following non-existence result shown by the second author in \cite[Proposition 2.2]{Spe03}.

\begin{corollary}\label{collectanea}
Suppose that $A = \begin{bmatrix} A_1 & 0 \\ T & A_2 \end{bmatrix}$, where  $A_1$ is $n_1 \times n_1$ matrix, $A_2$ is $n_2\times n_2$ matrix, and $T$ is $n_2 \times n_1$ matrix, such that $\left | \det A \right |>1$, $|\det A_1| > 1$, and $|\det A_2| < 1$. Suppose that $\Gamma$ is a full-rank lattice such that $\Gamma \cap (\{0\} \times \R^{n_2})$ has rank $n_2$ . Then there is no $(A, \Gamma)$ wavelet set.
\end{corollary}

\begin{proof}
Take a finite set $F \subset \Gamma \cap (\{0\} \times \R^{n_2})$ such that  span $V=\spa{F}=\{0\} \times \R^{n_2}$. 
Then 
\[
m_{n_2}(V \cap A^{-j}(\mathbf B(0, 1)))
=m_{n_2}((A_2)^{-j}(\mathbf B(0,1))) = c_{n_2} |\det A_2|^{-j},
\] 
where $c_d$ denotes the Lebesgue measure of $d$-dimensional unit ball. Since $|\det A_2|<1$, this yields the required conclusion by Theorem \ref{nexist}.\end{proof}

The proof of Theorem \ref{nexist} involves writing $A$ in real Jordan form, and relating the action of $A$ on balls to the action of corresponding matrices with positive eigenvalues on the same balls, and then to diagonal matrices. Lemma \ref{ite} is essentially a reformulation of \cite[Lemma 6.7]{ChesFuhr2020}. We give its proof for the sake of completeness.

\begin{lemma}\label{ite}
 Suppose that $B$ is an $n\times n$ invertible matrix. Then, there exists a positive constant $c$ and an $n\times n$ matrix $\tilde B$ such that all eigenvalues of $\tilde B$ are positive
 and
\begin{equation}\label{ite1}
B^j(\mathbf B(0,r/c)) \subset \tilde B^j(\mathbf B(0,r)) 
\subset B^j(\mathbf B(0,cr))
\qquad\text{for all }j\in \Z, \ r>0.
\end{equation}
\end{lemma}

\begin{proof}
 For $z\in \mathbb C$ we define $2\times 2$ matrix
\[
M_z=
\begin{bmatrix}
\operatorname{Re}(z) & \operatorname{Im}(z) \\
-\operatorname{Im}(z) & \operatorname{Re}(z) 
\end{bmatrix}.
\]
There exists an $n\times n$ invertible matrix $P$ such that $PBP^{-1}$ is in real Jordan normal form. That is, $PBP^{-1}$ is a block diagonal matrix consisting of real Jordan blocks $J_1, \ldots, J_k$. 
Suppose that a Jordan block $J_i$ corresponds to a complex eigenvalue $\lambda \in \mathbb C \setminus \R$,
\[
J_i=
\begin{bmatrix}
M_\lambda & M_1 & & & \\
& M_\lambda & M_1 & &\\
& & \ddots & \ddots & \\
& & & M_\lambda & M_1 \\
& & & & M_\lambda
\end{bmatrix}.
\]
We write $\lambda = |\lambda| \omega$, where $|\omega|=1$, and define matrices $R_i$ and $J_i'$ by
\[
R_i=
\begin{bmatrix}
M_{\ov \omega} & & & & & \\
& M_{\ov \omega} &  & &\\
& & \ddots &  & \\
& & & M_{\ov \omega} &  \\
& & & & M_{\ov \omega}
\end{bmatrix},
\qquad
J'_i = 
\begin{bmatrix}
M_{|\lambda|} & M_{\ov \omega} & & & \\
& M_{|\lambda|} & M_{\ov \omega} & &\\
& & \ddots & \ddots & \\
& & & M_{|\lambda|} & M_{\ov \omega} \\
& & & & M_{|\lambda|}
\end{bmatrix}
.
\]
Observe that $R_i$ is an orthogonal matrix, $R_i$ and $J_i$ commute, and $J_i'=R_i J_i$.
In the case a Jordan block $J_i$ corresponds to a real eigenvalue $\lambda<0$, we let $R_i=-I$ and $J_i'=-J_i$ if $\lambda<0$, and otherwise we set $R_i=I$ and $J_i'=J_i$.
Note that each block $J_i'$ has a single positive eigenvalue $|\lambda|$.

Define a matrix $\tilde B$ such that $P\tilde B P^{-1}$ is block diagonal with blocks $J_1',\ldots,J_k'$. Define a matrix $R$ to be block diagonal with blocks $R_1,\ldots, R_k$. By the construction, $R$ is an orthogonal matrix which commutes with $P B P^{-1}$, and 
\[
P\tilde B P^{-1} = R (P B P^{-1}) = (P B P^{-1})R.
\]
Consequently, we have for all $j\in \Z$,  $r>0$,
\[
(PBP^{-1})^j(\mathbf B(0,r)) = (P\tilde BP^{-1})^j(\mathbf B(0,r)).
\]
Hence,
\[
B^j(\mathbf B(0,r/||P||)) \subset (B^j P^{-1}) (\mathbf B(0,r)) = (\tilde B^j P^{-1}) (\mathbf B(0,r))
\subset \tilde B^j (\mathbf B(0,||P^{-1}||r)).
\] 
This proves \eqref{ite1} with $c= ||P^{-1}||\cdot ||P||$.
\end{proof}

The next lemma shows a basic estimate on the intersection of a subspace with dilates of a ball. For simplicity we formulate Lemma \ref{dil} for matrices with positive eigenvalues in Jordan normal form. 

\begin{lemma}\label{dil}
Suppose that $ B$ is an $n\times n$ matrix in Jordan normal form and all eigenvalues of $ B$ are positive. 
Let $V \subset \R^n$ be a subspace of dimension $d$. Then there exists 
a subset $\sigma \subset [n]$ of cardinality $d$ such that the following holds.
Let $P_\sigma$ be the coordinate orthogonal projection of $\R^n$ onto
\[
\R^\sigma=\spa\{ e_i: i\in \sigma\}.
\]
Then, the restriction $P_\sigma |_V$ is an isomorphism of $V$ and $\R^\sigma$ and
\begin{equation}\label{dil1}
 n^{-d/2} \le \frac{m_d(P_{\sigma}(V \cap { B}^{j}(\mathbf B(0,r)))) }{  m_d(( P_\sigma  B P_\sigma)^j (\mathbf B(0,r) ) )} \le n^{d/2}
\qquad\text{for all }j\ge 1, \ r>0.
\end{equation}
\end{lemma}

\begin{proof}
Without loss of generality we can assume that $ B$ has diagonal entries $\lambda_1,\ldots,\lambda_n$ satisfying $0<\lambda_1 \le \ldots \le \lambda_n$. We claim that there exists a basis $v_1,\ldots,v_d$ of $V$ and a sequence of integers $1 \le \sigma_1< \ldots <\sigma_d \le n$ such that
\begin{equation}\label{sp}
v_i -e_{\sigma_i} \in \spa\{ e_k: \sigma_{i}+1  \le k \le d \quad\text{ and }\quad k \ne \sigma_{i+1},\ldots,\sigma_{d}\} 
\qquad\text{for }i=1,\ldots,d.
\end{equation}
 Indeed, let $w_1,\ldots,w_d$ be any basis of $V$. Consider a $d\times n$ matrix which consists of vectors $w_1,\ldots, w_d$ treated as row vectors. Next, we perform a Gaussian elimination on this matrix to obtain a reduced row echelon form (only the first three rows shown):
\setcounter{MaxMatrixCols}{20}
 \[
 \begin{bmatrix}
 0 & \ldots & 0 & 1 & * & \ldots & * & 0 & * & \ldots  & * & 0 & * & \ldots
 \cr
 0 & \ldots & 0 & 0 & 0 &\ldots & 0 & 1 & * & \ldots & * & 0 & * & \ldots
 \\
  0 & \ldots & 0 & 0 & 0 &\ldots & 0 & 0 & 0 & \ldots & 0 & 1 & * & \ldots
 \end{bmatrix}
 \]
 Here, $*$ represents an arbitrary real value. Let $\sigma_i$ be the column which contains the leading entry $1$ in row $i$; all other entries in column $\sigma_i$ are zero. Let $v_i$ be the row vector given by row $i$. Since row operations preserve $V$, the vectors $v_1,\ldots, v_d$ form a basis of $V$ satisfying \eqref{sp}.

In addition, we assume that Jordan blocks $J$ of $B$ appear in  lower diagonal form (rather than the more common upper diagonal form)
\begin{equation}\label{jj}
J= \begin{bmatrix} \lambda & & & \\
1 & \ddots & & \\
& \ddots &  \lambda & \\
& &  1& \lambda 
\end{bmatrix}
\end{equation}
For a fixed $j\ge 1$, define the set 
\[
S=P_{\sigma}(V \cap { B}^{j}((-1,1)^n)).
\]
Since $P_\sigma(v_i) =e_{\sigma_i}$ for $i=1,\ldots, d,$ by identifying $\R^\sigma= \R^d$ we have
\[
S=\bigg \{(c_1,\ldots,c_d) \in \R^d: v=\sum_{i=1}^d c_i v_i \in B^j((-1,1)^n) \bigg\}.
\]
We shall estimate the $d$-dimensional Lebesgue measure of $S$ by Fubini's theorem
\begin{equation}\label{fub}
m_d(S)= \int_\R \ldots \bigg( \int_\R \ch_S(c_1,\ldots,c_d) dc_d \bigg) \ldots dc_1.
\end{equation}

First we consider an extreme case when the index set $\sigma$ is located in a single Jordan block $J$ of the form \eqref{jj} corresponding to an eigenvalue $\lambda$. Let $k$ be the size of $J$. If $j\ge k-1$, then $J^j((-1,1)^k)$ equals
\begin{equation}\label{jjj}
\textstyle
\ \{ (\lambda^j d_1, \lambda^j d_2 + j\lambda^{j-1} d_1, \ldots, \lambda^j d_k + j \lambda^{j-1} d_{k-1} + \ldots + \binom{j}{k-1} \lambda^{j-k+1} d_1 ): |d_1|, \ldots, |d_k| <1
\}.
\end{equation}
Suppose $(c_1,\ldots, c_d) \in S$. This implies that $c_1 \in (-\lambda^j, \lambda^j)$. Once such $c_1$ is fixed, then $c_2 \in (-\lambda^j, \lambda^j)+z_2$, for some $z_2$ depending on $c_1$. In general, once $c_1,\ldots,c_i$, $i<d$, are fixed, then by \eqref{jjj} we deduce that $c_{i+1} \in (-\lambda^j, \lambda^j)+z_{i+1}$, for some $z_{i+1}$ depending on $c_1,\ldots,c_i$. This shows that
\[
m_d(S)= (2\lambda)^d = m_d((P_\sigma B P_\sigma)^j((-1,1)^n)).
\]
In the general case, when $\sigma$ is spreads across different Jordan blocks, we can extend this argument, by grouping elements of $\sigma$ located in Jordan blocks, to deduce that
\[
m_d(S)= 2^d \prod_{i=1}^d \lambda_{\sigma_i} = m_d((P_\sigma B P_\sigma)^j((-1,1)^n)).
\]
By the fact that
\begin{equation}\label{tin}
(-1/\sqrt n,1/\sqrt{n} )^n\subset\mathbf B(0,1) \subset [-1,1]^n,
\end{equation}
and scaling, we deduce \eqref{dil1} for $j\ge 1$.
\end{proof}

As a corollary of Lemmas \ref{ite} and \ref{dil} we extend estimate \eqref{dil1} to arbitrary matrices.

\begin{corollary}\label{dila}
Suppose that $B$ is an $n\times n$ invertible matrix. Let $V \subset \R^n$ be a subspace of dimension $d$. Then, there exist positive constants $b$ and $C$ such that
\begin{equation}\label{dila0}
 b^j /C\le m_d(V \cap B^j(\mathbf B(0,1))) \le C b^j \qquad\text{for all }j\ge 1,
\end{equation}
where $m_d$ is $d$-dimensional Lebesgue measure on $V$.
\end{corollary}

\begin{proof}
 Let $\tilde B$ be a matrix with positive eigenvalues satisfying the conclusion of Lemma \ref{ite} corresponding to $B$. Then, we have
\[
V \cap {\tilde B}^j(\mathbf B(0,1/c)) \subset V \cap B^j(\mathbf B(0,1)) \subset V \cap {\tilde B}^j(\mathbf B(0,c)).
\]
Therefore, replacing matrix $B$ by $\tilde B$, without loss of generality we can assume that all eigenvalues of $B$ are positive.

Let $S$ be an invertible matrix such that $\check B= S B S^{-1}$ is in Jordan normal form. Note that
 \begin{equation}\label{dila2}
V \cap B^j(\mathbf B(0,1)) = V \cap S^{-1}  \check B^j S(\mathbf B(0,1)) = S^{-1} (SV \cap  \check B^j S(\mathbf B(0,1))).
\end{equation}
Since 
\[
\mathbf B(0,1/c') \subset S(\mathbf B(0,1)) \subset \mathbf B(0,c')  \qquad\text{where } c'=\max(||S^{-1}||,||S||),
\]
Lemma \ref{dil} implies that
\begin{equation}\label{dila4}
 n^{-d/2}(c')^{-d} \ \le \frac{m_d(P_{\sigma}(SV \cap { \check B}^{j} S (\mathbf B(0,1)))) }{  m_d(( P_\sigma  \check B P_\sigma)^j (\mathbf B(0,1) ) )} \le n^{d/2} (c')^d
\qquad\text{for all }j\ge 1.
\end{equation}
Here, $\sigma$ is a subset of $[n]$ of cardinality $d$ such that the restriction $P_\sigma |_{SV}$ is an isomorphism of $SV$ and $\R^\sigma$. Consequently,  $(P_\sigma \circ S)|_V$ is an isomorphism of $V$ and $\R^\sigma$. Hence, by \eqref{dila2}, there exists a constant $c''>0$ such that
\[
m_d(V \cap B^j(\mathbf B(0,1))) =c '' m_d(P_\sigma (SV \cap \check B^j S(\mathbf B(0,1)))) .
\]
 Combining this with  \eqref{dila4} yields \eqref{dila0}.
\
\end{proof}

We are now ready to give the proof of Theorem \ref{nexist}.

\begin{proof}[Proof of Theorem \ref{nexist}]
Let $B = A^{-1}$. By Corollary \ref{dila} there exist positive constants $b$ and $C$ such that \eqref{dila0} holds. By the hypothesis \eqref{sec} we necessarily have $b>1$.
By Lemma \ref{vp}
\[
\#|V \cap A^{-j}(\mathbf B(0,1)) \cap \Gamma )| \ge \frac{m_d(V \cap A^{-j}(\mathbf B(0,1))) }{2^d m_d(V/(V\cap \Gamma))}.
\]
Hence, for some constant $C'>0$ we have
\[
\sum_{j=1}^\infty \frac{1}{\#| A^{-j}(\mathbf B(0,1))\cap \Gamma|} \le 
 \sum_{j=1}^\infty  \frac{2^d m_d(V/(V\cap \Gamma))}{m_d(V \cap A^{-j}(\mathbf B(0,1))) }
\le C' \sum_{j=1}^\infty b^{-j}<\infty.
\] 
By Theorem \ref{nece} there is no $(A,\Gamma)$ wavelet set.
\end{proof}

\section{Applications of main results}\label{idk}

In \cite{DaiLarSpe97}, it was shown that if all eigenvalues of a matrix $A$ have modulus strictly larger than one, then for every full rank lattice $\Gamma$ there exists an $(A, \Gamma)$ wavelet set. Various wavelet sets for matrices with some eigenvalues greater than one and some equal to one have been constructed in the literature. In our first application of Theorem \ref{mainTheorem}, we show that we are always able to construct such wavelet sets. 

Recall that a matrix $A$ is {\emph{unipotent}} if all of its eigenvalues are 1. The key added ingredient is the following result due to Margulis \cite[Theorem 1]{Margulis71}, see also \cite{Dani}. 

\begin{theorem}\label{marg71}
Let $\Gamma$ be a full rank lattice in $\R^n$ and let $U_t$ be a one parameter group of unipotent matrices. There exists $\delta > 0$ such that 
\[
\sup\{t \in \R: \bfb(0, \delta) \cap U_t \Gamma = \{0\}\} = \infty.
\]
\end{theorem}
We will apply Theorem \ref{marg71} in the case $U_t = J^t$, where $J$ is the  block of the real Jordan decomposition of a matrix corresponding to eigenvalue 1.

\begin{theorem}\label{eigenbiggerone}
Let $A$ be an $n\times n$ matrix such that $\left | \det A \right |>1$ and all eigenvalues of $A$ are greater than or equal to one in modulus. Then, for every full rank lattice $\Gamma$, there exists an $(A, \Gamma)$ wavelet set.
\end{theorem}

\begin{proof}
By Lemma \ref{ite}, there exists a constant $c$ and an $n \times n$ invertible matrix $\tilde A$ such that all of the eigenvalues of $\tilde A$ are positive and such that for all $j\in \Z$ and all $r > 0$,
\[
A^j(\bfb(0, r/c)) \subset \tilde A^j(\bfb(0, r)) \subset A^j(\bfb(0, cr)).
\]

By a change of basis, we can assume that $\tilde A$ is in Jordan form. 
  Write $\tilde A = \begin{bmatrix}
  A_1&0\\
  0&A_2
  \end{bmatrix}$
  where all eigenvalues of $A_1$ are larger than 1, and all eigenvalues of $A_2$ are equal to 1. Let $T$ denote the $n \times n$ matrix $T = \begin{bmatrix}
  I&0\\
  0&A_2
  \end{bmatrix}.$
If $A_2$ is the identity matrix, then there exists $\epsilon > 0$ such that $\tilde A^{-j}\left(\bfb(0, \epsilon)\right) \cap \Gamma = \{0\}$ for all sufficiently large $j$. Therefore, $A^{-j}(\bfb(0, \epsilon/c)) \cap \Gamma = \{0\}$ for sufficiently large $j$ as well, so $(A, \Gamma)$ wavelet sets exist by Theorem \ref{mainTheorem}.

If $A_2$ contains a non-trivial Jordan block, then by assumption, $T$ generates a one parameter unipotent group, so by Theorem \ref{marg71}, there exists an $\epsilon > 0$ and $n_1 < n_2 < \cdots$ such that $T^{-n_k}(\bfb(0, \epsilon)) \cap \Gamma = \{0\}$. For large $k$, $\tilde A^{-n_k}(\bfb(0, \epsilon)) \subset T^{-n_k}(\bfb(0, \epsilon)),$ so there exist $n_1 < n_2 < \cdots$ such that $\tilde A ^{-n_k}(\bfb(0, \epsilon)) \cap \Gamma = \{0\}$. Since $A^{-n_k}(\bfb(0, \epsilon/c)) \subset \tilde A^{-n_k} (\bfb(0, \epsilon)),$ $(A, \Gamma)$ wavelet sets exist by Theorem \ref{mainTheorem}.
\end{proof}

\begin{remark} Theorem \ref{eigenbiggerone} is optimal in the sense that it identifies the largest class of matrices $A$ with $\left | \det A \right |>1$ for which $(A, \Gamma)$ wavelet set exists regardless of the choice of the lattice $\Gamma$. Indeed, if one of eigenvalues of $A$ is less than $1$, then we consider  the invariant subspace $V$, which is the linear span of eigenspaces of $A$ corresponding to eigenvalues $|\lambda|<1$. Then, we choose a full rank lattice $\Gamma \subset \R^n$ such that the rank of $\Gamma \cap V$ equals $d=\dim V$. Since $\lim_{j\to \infty} m_d(V \cap A^{-j}(\mathbf B(0, 1)))= \infty$, by Theorem \ref{nexist} we deduce that there is no $(A,\Gamma)$ wavelet set.
\end{remark}

The second application of Theorem \ref{mainTheorem} is a generalization of the main result in \cite{IonWan06}. We will need to use the inequality between the measure of an ellipsoid intersected with a subspace and the measure of the ellipsoid and lengths of its principal semi-axes.

\begin{lemma}\label{ellipsoidsection}
There exists a constant $C = C(n)$ such that whenever $\mathcal{E} \subset \R^n$ is an ellipsoid with lengths of principal semi-axes given by $\lambda_1 \ge \cdots \ge \lambda_n > 0$ and $V$ is a subspace of $\R^n$ of dimension $k$,
\[
m_k(V \cap \mathcal{E}) \le C \prod_{i = 1}^k \lambda_i.
\]
\end{lemma}

\begin{proof}
  Fix $k$. We first note that $\mathcal{E} \cap V$ is an ellipsoid with $k$ non-zero lengths of principle semi-axes given which we denote by $\omega_1 \ge \cdots \ge \omega_k > 0$. Let $\{v_1, \ldots, v_n\}$ be an orthonormal basis of $\R^n$ such that the first $k$ vectors form an orthonormal basis for $V$. Let $A$ be an $n \times n$ positive definite matrix written in terms of $\{v_1, \ldots, v_n\}$ such that $\mathcal{E} = {\text {conv}} \{x\in \R^n: \langle Ax, x\rangle = 1\}$. Note that $\mathcal{E} \cap V = {\text {conv}} \{x \in \R^n: \langle PAPx, x\rangle = 1\}$, where $P$ is the orthogonal projection of $\R^n$ onto $V$. Let $\lambda_j(A)$ denote the $j$th smallest eigenvalue of $A$, and $\lambda'_j(B)$ denote the $j$th smallest non-zero eigenvalue of $B=PAP$. By the inclusion principle, \cite[Theorem 4.3.15]{HornJohnson85} $\lambda_i(A) \le \lambda'_i(PAP) \le \lambda_{i + n - k}(A)$. Therefore, $\omega_i = \lambda'_{i}(PAP)^{-1/2} \le \lambda_i(A)^{-1/2} = \lambda_i$. In particular,
  \[
  m_k(V \cap \mathcal{E}) = c_k \prod_{i = 1}^k \omega_i \le c_k \prod_{i = 1}^k \lambda_i,
  \]
  where $c_k = \frac {\pi^{k/2}}{\Gamma(k/2 + 1)}$ is the Lebesgue measure of $k$-dimensional unit ball. The lemma follows by letting $C = \max(c_1, \ldots, c_n)$.
\end{proof}

\begin{theorem}\label{diagiw}
Let $A$ be an $n \times n$ diagonal matrix with $\left | \det A \right | > 1$ and with eigenvalues arranged so that $|\lambda_1| \ge |\lambda_2| \ge \cdots \ge |\lambda_{n - 1}| > 1 > |\lambda_n|$. Assume in addition that $|\lambda_n \lambda_{n - 1}| \ge 1$. Let $\Gamma \subset \R^n$ be a full rank lattice. Then, there exists an $(A, \Gamma)$ wavelet set if and only if $\Gamma \cap \spa(e_n)= \{0\}$, where $e_n$ is the last standard unit basis vector.
\end{theorem}

\begin{proof}
We start by showing the ``only if'' part.  The proof follows from the following two claims.
  
  \begin{claim}\label{c55}
  Let ${\mathcal{F}}$ denote the collection of subspaces of $\R^n$ of dimension at least two. For every $R > 0$, the sequence $(\gamma_j)_{j = 1}^\infty$ defined by 
\[
\gamma_j = \sup\{m_k\left(V \cap A^{-j}(\bfb(0, R))\right): V \in \mathcal{F}, k = \dim V\}
\]
is bounded.
\end{claim}

\begin{proof}[Proof of Claim] 
  Let $V$ be a subspace of $\R^n$ of dimension $k \ge 2$. The intersection $V \cap A^{-j}(\bfb(0, R))$ is an ellipsoid. By Lemma \ref{ellipsoidsection}, the $k$-dimensional volume of $A^{-j}(\bfb(0, R)) \cap V$ is less than or equal to $C |\lambda_{n - 1} \lambda_n|^{-j},$ which is bounded.
\end{proof}

\begin{claim} \label{c56}
If $\lim_{j\to \infty} \#|A^{-j}(\bfb(0, R)) \cap \Gamma| = \infty$, then 
\[
\limsup_{j \to \infty}\dim {\rm{span}} (A^{-j}(\bfb(0, R)) \cap \Gamma) \ge 2
\]
\end{claim}

\begin{proof}[Proof of Claim]
Choose a cylinder  of the form $S = \bfb_{n - 1}(0, R) \times [-R, R]$ and notice that $\bfb(0, R) \subset S$.
We prove the following formally stronger statement than the claim: 
\[
L = \limsup_{j \to \infty}\dim  {\rm{span}} (A^{-j}(S) \cap \Gamma) \ge 2.
\]

We first note that $L \ge 1$ since there is a constant $K$ such that for $j \ge K$, $\#|A^{-j}(S) \cap \Gamma| > 1$. Assume for the sake of contradiction that $L = 1$, and redefine $K$ so that for $j \ge K$, $\dim {\rm{span}} \left(A^{-j}(S) \cap \Gamma\right) = 1$. For fixed $j \ge K$, let $v$ be the element of $A^{-j}(S) \cap \Gamma$ of smallest norm, so that $A^{-j}(S) \cap \Gamma \subset v \Z$. There exists a largest $k \ge j$ such that $v \in A^{-k}(S)$ since  $\Gamma \cap \{t e_1: t\in \R\} = \{0\}$. There exists non-zero $v_2 \in A^{-k - 1}(S) \cap \Gamma$, and the last coordinate of $v_2$ must be in the interval $[|\lambda_n|^{-k} R, |\lambda_n|^{-k - 1} R]$. This implies that $v_2/|\lambda_n| \not\in A^{-k - 1}(S)$, and $\#|A^{-k - 1}(S) \cap \Gamma| \le 2/|\lambda_n|$. Since $j$ was an arbitrary integer larger than $K$, this contradicts the hypothesis that $\lim_{j\to \infty} \#|A^{-j}(\bfb(0, R)) \cap \Gamma| = \infty$.
\end{proof}

We turn to the proof of the theorem. Assume $\Gamma \cap \{t e_n: t\in \R\} = \{0\}$; we wish to show that there exists an $(A, \Gamma)$ wavelet set by showing that $\liminf_{j \to \infty} \#|A^{-j}(\bfb(0, R))\cap \Gamma| < \infty$, which implies $\sum \frac {1}{\#\left|A^{-j}(\bfb(0, 1)) \cap \Gamma\right |} = \infty$. Assume to the contrary that this limit is infinite. 

By Claim \ref{c56} there exists an increasing sequence $(j_k)_{k=1}^\infty$ such that 
\begin{equation}\label{c57}
\dim {\rm{span}} (A^{-j_k}(\bfb(0, R)) \cap \Gamma) \ge 2 \qquad\text{for all }k\in\N.
\end{equation}
Let $V_k$ be the span of $A^{-j_k}(\bfb(0, R)) \cap \Gamma$ and let $\Lambda_k = V_k \cap \Gamma$. Let $F_k$ be a fundamental domain of $\Lambda_k$. Let $n_k=\dim V_k$. By  Lemma \ref{vp}, Claim \ref{c55}, and \eqref{c57} we have
\begin{eqnarray*}
\#|A^{-j_k}(\bfb(0, R)) \cap \Gamma | &=& \#|A^{-j_k}(\bfb(0, R)) \cap V_k \cap \Lambda_k|\\
&\le& \frac{3^{n_k} n_k!}{2^{n_k} |F_k|} m_{n_k} \left(A^{-j_k}(\bfb(0, R)) \cap V_k \right)\\
&\le& C \gamma_{j_k}.
\end{eqnarray*}
The last inequality is due to a positive lower bound on the measure of the fundamental domains of non-zero sublattices of a fixed lattice. Since $(\gamma_{j_k})$ is bounded, this is a contradiction to our hypothesis that $\liminf_{j \to \infty} \#|A^{-j}(\bfb(0, R))\cap \Gamma| = \infty$. Hence, by Theorem \ref{mainTheorem} $(A, \Gamma)$ wavelet sets exist.

For the other direction, if $\Gamma \cap \{t e_n: t\in \R\} \not= \{0\}$, then we let $\Lambda = \{t e_n: t\in \R\} \cap \Gamma$ and $V = {\rm {span}} (F).$ Since 
\[
\lim_{j \to \infty} m_1(V \cap A^{-j} \bfb(0, 1)) = \infty,
\]
by Theorem \ref{nexist}, there is no $(A, \Gamma)$ wavelet set.
\end{proof}

\begin{remark}
The above proof shows that for diagonal matrices $A$ satisfying the assumptions in Theorem \ref{diagiw}, conditions \eqref{ctwo}, \eqref{cthree}, and \eqref{cfour} in the Introduction are all equivalent. This equivalence can also be deduced from the result of Weiss \cite[Theorem 5.2]{Wei04} on divergent trajectories.

\end{remark}

Our third application of the main theorem in this paper is the following.

\begin{theorem}\label{lce}
Let $A$ be $n\times n$ matrix with integer entries and $\left | \det A \right | \ge 2$. Then, the pair $(A,\Z^n)$ satisfies the lattice counting estimate
\begin{equation}\label{lce1}
\#| \Z^n \cap A^j(\mathbf B(0,r))| \le C \max(1,r^n) \max(1, \left | \det A \right |^j) \qquad\text{for all }j\in \Z, r>0,
\end{equation}
for some positive constant $C$ depending only on $n$.
In particular, there exists $(A,\Z^n)$ wavelet set.
\end{theorem}

\begin{proof}
For $j<0$ we have $A^{-j}\Z^n  \subset \Z^n$. Hence, 
\[
\#| \Z^n \cap A^j(\mathbf B(0,r))| = \#|A^{-j}\Z^n \cap \mathbf B(0,r) | \le \#| \Z^n \cap \mathbf B(0,r)| = C(r) \le C \max(1,r^n).
\]
For $j>0$ we use the fact that each coset of the quotient group $A^{-j}\Z^n /\Z^n$ intersects $[0,1)^n$ exactly once. Hence, for any $k\in \Z^n$,
\[
\#| A^{-j} \Z^n \cap (k+[0,1)^n)| = \#|A^{-j}\Z^n /\Z^n| = \#|\Z^n/(A^j\Z^n)| = \left | \det A \right |^j.
\]
Consequently,
\[
\begin{aligned}
\#| \Z^n \cap A^j(\mathbf B(0,r))| & = \#|A^{-j}\Z^n \cap \mathbf B(0,r) | 
\\
& \le \sum_{k\in \Z^n, \ |k| < r+\sqrt{n}}
 \#|A^{-j}\Z^n \cap (k+[0,1)^n)| 
 \\
 & \le \#| \Z^n \cap \mathbf B(0,r+\sqrt{n}) | \left | \det A \right |^j  = C(r+\sqrt{n}) \left | \det A \right |^j.
\end{aligned}
\]
This proves \eqref{lce1}. The existence of $(A,\Z^n)$ wavelet set follows then by Theorem \ref{mainTheorem}.
\end{proof}

\section{Examples}\label{examples}

In this section we provide two examples in three dimensions which relate to the theorems proven above. As always, we are assuming that $\left | \det A \right | > 1$.

The following example shows that the characterization of wavelet sets in Theorem \ref{diagiw} requires the condition on the product of eigenvalues. 
Namely, if the product of the smallest two eigenvalues is less than one, then there are pairs $(A, \Gamma)$ such that no wavelet set exists despite that $\Gamma \cap \spa (e_n) = \{0\}$.

\begin{example}\label{obvious}
Let $A$ be a $3\times 3$ diagonal matrix with entries $\lambda_1$, $\lambda_2$, $\lambda_3$. Let $\Gamma$ be a full rank lattice in $\R^3$. Suppose that
\begin{enumerate}[(i)]
\item $|\lambda_1| \ge |\lambda_2| \ge |\lambda_3|$,
\item $\left | \det A \right |= |\lambda_1\lambda_2\lambda_3|>1$,
\item $|\lambda_1\lambda_3|<1$,
\item there exists $\gamma_1,\gamma_2 \in \Gamma$ such that
\[
\spa (\gamma_1,\gamma_2) = \spa(ae_1+be_2,e_3),\qquad\text{where }a,b\ne 0.
\]
\end{enumerate}
Then, there is no $(A,\Gamma)$ wavelet set.
\end{example}

\begin{proof}
Take $F=\{\gamma_1,\gamma_2\}$, $V = \spa F$, and observe that 
\[
m_2(V \cap A^{-j}((-1,1)^3)) = m_2( V \cap \prod_{i=1}^3 (-|\lambda_i|^{-j},|\lambda_i|^{-j})) \ge |\lambda_1 \lambda_3|^{-j} \to \infty
\]
as $j\to \infty$. Hence, \eqref{sec2} holds. By Theorem \ref{nexist}, there is no $(A,\Gamma)$ wavelet set.
\end{proof}

Consider the condition that for some subset $F \subset \Gamma$ we have
\begin{equation}\label{sec2}
\lim_{j\to \infty} m_d(V \cap A^{-j} \mathbf B(0,R)) = \infty,
\qquad\text{where }V= \spa F, \ d=\dim V.
\end{equation}
It was shown in Theorem \ref{nexist} that this condition implies that $(A, \Gamma)$ wavelet sets do not exist. In Example \ref{lcc}, we provide an example of a matrix $A$ and a lattice $\Gamma$ such that condition \eqref{sec2} is fails, yet no $(A, \Gamma)$ wavelet set exists. This shows that \eqref{sec2} is not equivalent to the non-existence of wavelet sets. 

 The easiest path to showing that condition \eqref{sec2} is not equivalent to the characterizing condition is to use the following result of  Khinchine, as stated in \cite{Mos10}.

\begin{theorem}\label{khinchine}
Suppose that $\psi(j) \to 0$ as $j \to \infty$. There exist numbers $\alpha$ and $\beta$ such that 
for sufficiently large $j$, there is an integer solution $(a,b)\in \Z^2$ of the system
\[
\|a \alpha + b \beta\| < \psi(n), \qquad 0 < \max(|a|, |b|) < j,
\]
where $\| \cdot \|$ denotes the distance to the nearest integer.
\end{theorem}

We will also use the following lemma on codimension $1$ sections of an ellipsoid.

\begin{lemma}\label{cross-section} Let $A$ be an $n \times n$ invertible matrix and let $\xi \in \R^n$ such that $\|\xi\| = 1$. Then, 
\[
m_{n - 1}\left(A(\bfb(0, 1)) \cap \xi^\perp \right) = C_n \frac{|\det A|}{\|A^T\xi\|},
\qquad\text{
where }C_n = \frac{\pi^{(n-1)/2}}{\Gamma((n + 1)/2)}.
\]

\end{lemma}

\begin{proof}
Let $B = \bfb(0, 1) \subset \R^n$. By \cite[Theorem 2.2]{KolYas08}, the Minkowski norm $||\cdot||_D$ associated to any centrally symmetric convex body $D \subset \R^n$ satisfies
\[
m_{n - 1}\left(D \cap \xi^\perp\right) = \frac{1}{\pi(n-1)} (\|\cdot\|_D^{-n+1})^\wedge(\xi).
\]
Moreover, the Euclidean norm $||\cdot||=||\cdot||_B$ satisfies
\[
(\| \cdot \|^{-n + 1})^\wedge(\xi) = \frac{2\pi^{(n+1)/2}}{\Gamma((n-1)/2)} \|\xi\|^{-1}.
\]
Let $C_n$ be as in the lemma statement. Then, we have
\begin{align*}
(\|\cdot \|_{A(B)}^{-n + 1})^\wedge(\xi) &= (\|A^{-1}(\cdot)\|_B^{-n + 1})^\wedge(\xi) = \frac{(\|\cdot\|_B^{-n + 1})^\wedge(A^T\xi)}{|\det(A^{-1})|}= C_n \frac{|\det A|}{\|A^T\xi\|},
\end{align*}
as desired.
\end{proof}

\begin{example}\label{lcc}
There exists a lattice $\Gamma \subset \R^3$ and an invertible, diagonal matrix $A$ with $|\det A| > 1$ such that no $(A,\Gamma)$ wavelet sets exist 
and such that for every $R > 0$ and every subset $F \subset \Gamma$
\[
m_d(V \cap A^{-j}(\bfb(0, R)) ) \to 0,
\]
where $V$ is the span of $F$ and $d=\dim V$.
\end{example}

\begin{proof}
Let $B = [-1, 1]^3$. For $\psi(j) = 11^{-j}$, let $\alpha$ and $\beta$ as in Khinchine's Theorem. 
Let 
\begin{equation}\label{exa}
A = \begin{bmatrix} 10 &0 & 0 \\ 0 & \frac 12 & 0 \\ 0 & 0 & \frac 12 \end{bmatrix}.
\end{equation}
Let $\Gamma$ be the lattice given by 
\begin{equation}\label{exb}
\Gamma = \spa_{\Z}\{(1,0,0), (\alpha,1,0), (\beta,0,1)\}.
\end{equation}
 For large $j$, there exist integers $a, b$ and $N$ such that $|N + a \alpha + b\beta| < 11^{-j}$ and $\max(|a|, |b|) < j$. Therefore, 

\begin{align*}
N (1,0,0) + a (\alpha, 1, 0) + b(\beta, 0, 1) & \in [-11^{-j}, 11^{-j}] \times [-j,j]^2  \\
& \subset \frac{10^j}{11^j} A^{-j}(B)
\end{align*}
It follows that $\#| A^{-j}(B) \cap \Gamma| \ge \frac{11^j}{10^j}$ for large $j$. Hence, by Theorem \ref{nece} no $(A,\Gamma)$ wavelet sets exist. 

Next, let $R > 0$ and $F \subset \Gamma$. We need to show that
\[
m_V(V \cap A^{-j}(\bfb(0, R) ) \to 0,
\]
where $V$ is the span of $F$. If $\dim V = 1$, then this is follows from the fact that there are no elements of $\Gamma$ in ${\rm {span}} \{(0, 1, 0), (0, 0, 1)\}$.

If $\dim V = 2$, let $\xi$ be a norm-one vector such that $V = \xi^\perp$. Since $\xi \not\in {\rm {span}} \{(1, 0, 0)\}$, by Lemma \ref{cross-section} we have
\[
m_V(A^{-j}(\bfb(0, R)) \cap \xi^\perp) = C\frac{ \|A^{-j} \xi\|^{-1}}{|\det(A^{j})|} \to 0 \qquad\text{as }j\to \infty.
\qedhere\]
\end{proof}

\bibliographystyle{plain}
\bibliography{bibliography.bib}

\end{document}